\documentclass[12pt]{article}
\usepackage[numbers,sort&compress]{natbib}
\usepackage{color}
\usepackage{threeparttable}
\usepackage[english]{babel}
\usepackage{t1enc}
\usepackage{amsthm,amsmath,amssymb}
\usepackage{mathrsfs}
\usepackage[mathscr]{eucal}
\usepackage[latin1]{inputenc}
\usepackage[english]{babel}
\usepackage{latexsym}
\usepackage{graphicx}
\usepackage{lineno}
\usepackage{booktabs}
\usepackage{multirow}
\newtheorem{lemma}{Lemma}

\def \T{\textup{T}}


\title{The smallest pair of cospectral cubic graphs with different chromatic indexes}
\author{\small Zhidan Yan\quad\quad Wei Wang\thanks{Corresponding author. Email: wangwei.math@gmail.com}
\\
{\footnotesize School of Mathematics, Physics and Finance, Anhui Polytechnic University, Wuhu 241000, P. R. China}
}
\date{}
\begin{document}
 \maketitle
 \begin{abstract}
Using an exhaustive search on cubic graphs of order 16, we find a unique cospectral pair with different chromatic indexes. This example indicates that the chromatic index of a regular graph is not characterized by its spectrum, which answers a question recently posed in [O.~Etesami, W.~H.~Haemers, On NP-hard graph properties characterized by the spectrum, Discrete Appl. Math., 285(2020)526-529]. We prove that any orthogonal matrix representing the similarity between the two adjacency matrices of the cospectral pair  cannot be rational. This implies that the cospectral pair cannot be obtained using the original GM-switching method or its generalizations based on rational orthogonal matrices.\\


\end{abstract}
\section{Introduction}
\label{intro}
A graph property $\mathcal{P}$ is \emph{determined  by its adjacency spectrum} (DS) if there exist no cospectral graphs such that one has the property $\mathcal{P}$ and the other one not. For a given property $\mathcal{P}$, we shall refer to  $\mathcal{P}$ as a DS-property if $\mathcal{P}$ is DS, otherwise we refer to  $\mathcal{P}$ as a non-DS-property.  Being bipartite and being regular are two DS-properties, while being connected is a non-DS-property. It is an interesting problem to differentiate DS-properties and non-DS-properties and this problem becomes more interesting if we restrict to regular graphs~\cite{haemers2020DMGT}. Recently, many graph properties are proved to be non-DS by constructing the  desired cospectral pairs of graphs which are usually regular.  For example, Bl\'{a}zsik, Cummings and Haemers~\cite{blazsik2015DM} constructed a pair of cospectral $k$-regular graphs, where one has a perfect matching and the other one not. Other non-DS-properties include being  distance-regular~\cite{haemers1996LAA}, having a given diameter~\cite{haemers1995LMA}, having a given vertex-connectivity or edge-connectivity~\cite{haemers2020DMGT},  having a given zero forcing number \cite{abiad2021arxiv}, and admitting a Hamiltonian cycle~\cite{filar2005AMSG,etesami2020DAM,liu2020LAA}.

In searching of NP-hard properties determined by spectrum, Etesami and Haemers~\cite{etesami2020DAM}  asked whether there exist two cospectral regular graphs with different chromatic indexes. By Vizing theorem, the chromatic index of a $k$-regular graph is either $k$ or $k+1$. By a result of Holyer~\cite{holyer1981SJC}, it is NP-complete to determine the chromatic index for an arbitrary graph and even for a cubic graph. In~\cite{etesami2020DAM}, the authors tended to believe that there exists a pair of cospectral $k$-regular graphs for which one is $k$-edge-colorable and the other is  $(k+1)$-edge-colorable (but not $k$-edge-colorable).  However, results from~\cite{cioabatoappearAMC} suggest that the desired cospectral pair may be rare and hard to find.

In 2005,  Filar, Gupta and Lucas~\cite{filar2005AMSG} began to consider the relationship between cospectrality and Hamiltonicity. They reported a pair of  cubic graphs with 20 vertices, where one is Hamiltonian and the other is not;  see Figure \ref{ex20}.  This means that Hamiltonicity is a non-DS-property, even restricted to regular graphs. It is clear that any cubic graph with a Hamiltonian cycle is $3$-edge-colorable, but the converse is known to be not true in general.  Thus, we know that the left graph in Figure 1 is $3$-edge-colorable but it is not clear whether the right graph is $3$-edge-colorable or not. Using Mathematica on a personal computer, it takes nearly 30 seconds to obtain the chromatic polynomial of the line graph for the right one.  We do not write down the huge polynomial exactly but only stress that it factors as
 $(x-3) (x-2)^2 (x-1) x \left(x^{25}-52 x^{24}+\cdots-8285534208\right)$. Note that the value of the polynomial is zero at $x=3$, that is, there is no proper $3$-vertex-coloring for the line graph. In other words,  the right graph is not $3$-edge-colorable. Therefore, this pair of graphs verifies the  aforementioned guess of Etesami and Haemers~\cite{etesami2020DAM}.
  \begin{figure}[htbp]
 	\centering
 	\includegraphics[height=4.6cm]{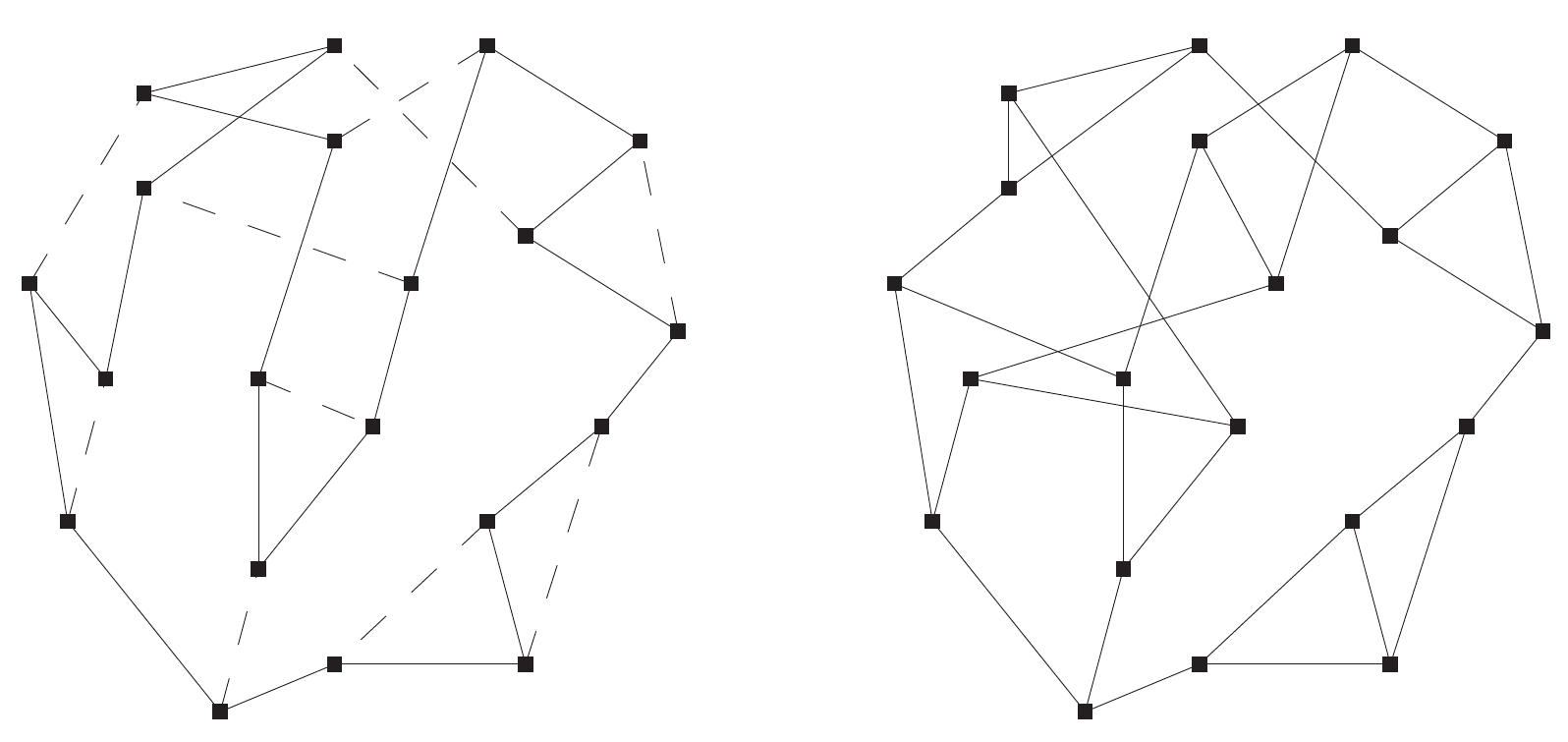}
 	\caption{A pair of cospectral cubic graphs with and without a Hamiltonian cycle~\cite{filar2005AMSG}}
 	\label{ex20}
 \end{figure}

In this note, we give a pair of smallest cospectral cubic  graphs with different chromatic indexes. We find this pair through an exhaustive search on cubic graphs with at most 16 vertices. All these cubic graphs are generated using the \emph{nauty} package \cite{mckay2014JSC} in SageMath~\cite{sage}. There are no cospectral cubic graphs with at most 12 vertices. For cubic graphs of order 14, there are exactly 3 pairs of cospectral graphs. All these 6 graphs are $3$-edge-colorable.  For cubic graphs of order 16, there are exactly 41 cospectral pairs together with one cospectral triple. The three graphs in the unique cospectral triple are all $3$-edge-colorable. Among the 41 cospectral pairs,  only one pair contains two graphs with different chromatic indexes. We shall describe the unique cospectral pair in Section \ref{ex}. Interestingly enough, one graph is closely related to the Petersen graph. This observation leads to a very simple proof for the nonexistence of a proper $3$-edge-coloring.

We prove that the cospectral pair found in this note cannot be constructed from the well-known GM-switching method, which was originally introduced in~\cite{godsil1982AM}. Indeed we prove an even stronger result which says that any orthogonal matrix connecting the two adjacency matrices cannot be rational.

Lastly, we mention that the Mathematica codes for all computations in this note can be found at https://github.com/wangweiAHPU/cospectral\_cubic\_graphs.

\section{The smallest example}\label{ex}
 The two graphs shown in Figure \ref{ex16} are cospectral with characteristic polynomial
$$(x-3) x (x+2) \left(x^2-2\right) \left(x^2-x-3\right) \left(x^3-4 x-2\right) \left(x^3-4 x+1\right) \left(x^3+2 x^2-2 x-2\right).$$
 \begin{figure}[htbp]
	\centering
	\includegraphics[height=4.6cm]{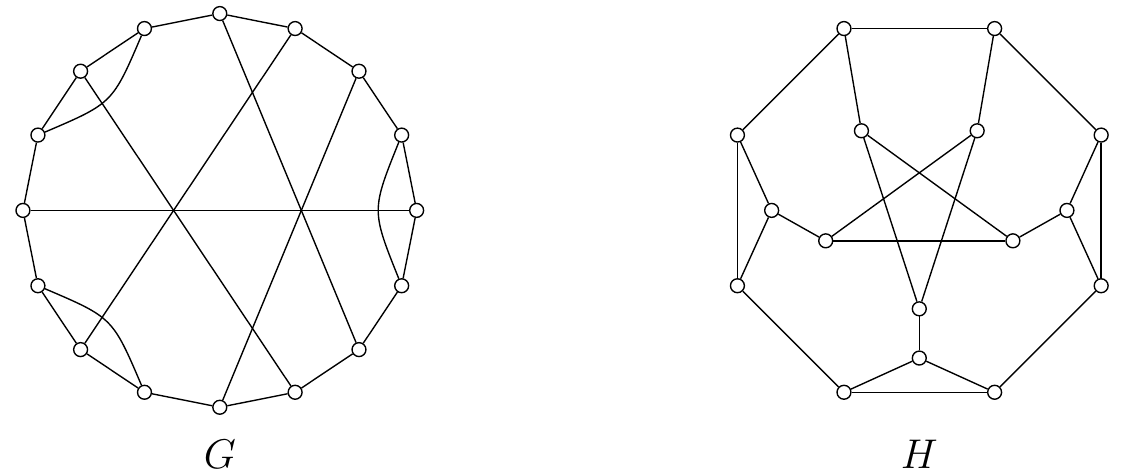}
	\caption{A pair of cospectral cubic graphs with different chromatic indexes}
	\label{ex16}
\end{figure}
 By the drawing of $G$, we easily see that $G$ is $3$-edge-colorable as $G$ contains a Hamiltonian cycle. We claim that the right graph $H$ is not $3$-edge-colorable. From the drawing of $H$, one  easily sees that $H$ can be obtained from the Petersen graph  by replacing each vertex  in a path $(u,v,w)$ with a triangle; see Figure \ref{petersen}.  Moreover, it is not difficult to see that there is a natural 1-1 correspondence  between the collection of all proper $3$-edge-colorings of $H$ and that of the Petersen graph. But it is well-known that the Petersen graph is not $3$-edge-colorable (see \cite{naserasr2003,volkmann} for example). Thus,  $H$ is not $3$-edge-colorable.

  \begin{figure}[htbp]
 	\centering
 	\includegraphics[height=4.2cm]{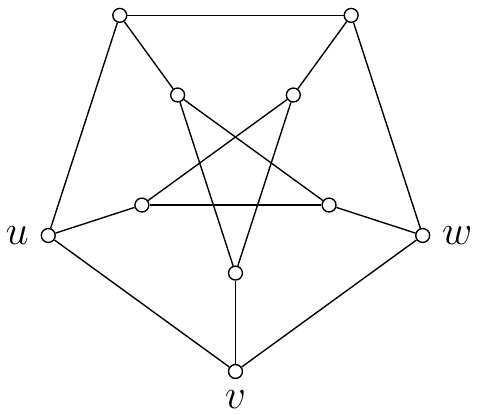}
 	\caption{The Petersen graph}
 	\label{petersen}
 \end{figure}

 We note that $H$ is not Hamiltonian as $H$ is not $3$-edge-colorable. Thus this pair of cospectral graphs also explains that the Hamiltonicity of a graph can  not be determined by the spectrum.

Let $Q$ be any orthogonal matrix such that $A(G)=Q^\T A(H)Q$, where $G$ and $H$ are the cospectral pair as shown in Figure \ref{ex16}, and $A(\cdot)$ denotes the adjacency matrix of a graph.  We shall prove that $Q$ cannot be a rational matrix. We need a simple lemma, which is a slight variant of \cite[Theorem 7]{liu2017EJC}. In
\cite{liu2017EJC}, Liu and Wang gave a family of cospectral but nonregular graphs such that no rational orthogonal matrix exists for the similarity between the adjacency matrices.

For a vector $\xi\in \mathbb{R}^n$, we use $||\xi||$ to denote the Euclidean norm of $\xi$.
  \begin{lemma}\label{intirration}
	Let $G_1$ and $G_2$ be a pair of cospectral graphs with a common simple integral eigenvalue $\lambda$. Let  $\xi$ (resp $\eta$) be an integral eigenvector of $A(G_1)$ (resp $A(G_2)$) corresponding to $\lambda$. If  $||\eta||/|| \xi||$ is irrational, then there exists no rational orthogonal matrix $Q$ such
	that $A(G_1) = Q^\T A(G_2)Q$.
\end{lemma}
\begin{proof}
	Suppose to the contrary that there exists a rational orthogonal $Q$ such that  $A(G_1) = Q^\T A(G_2)Q$. Then $QA(G_1)=A(G_2)Q$ and hence $A(G_2)Q\xi=QA(G_1)\xi=\lambda Q\xi$. This means that $Q\xi$ is an eigenvector of $A(G_2)$ corresponding to $\lambda$.  As $\lambda$ is a simple eigenvalue, we must have $\eta=rQ\xi$ for some $r\in\mathbb{R}$. Since $\xi,\eta,Q$ are all rational, we see that $r$ is rational. As $Q$ is orthogonal, we have $||\eta||=|r|\cdot||Q\xi||=|r|\cdot||\xi||$ and hence $||\eta||/|| \xi||=|r|$, which is rational. This contradicts the assumption of this lemma and hence completes the proof.
\end{proof}
For the two graphs $G$ and $H$ shown in Figure \ref{ex16},
we prove that there is no rational  orthogonal matrix $Q$ such that $A(G)=Q^\T A(H)Q$. It can be easily verified that $0$ is a common simple eigenvalue of $A(G)$ and $A(H)$. Moreover  $\xi=(0, -1, 0, 1, -1, 0, 0, 0, 1, 1, -1, 0, 0,$ $-1, 0, 1)^\T$
and $\eta=(2, 1, -1,$ $ -1, 1, 0, 0, -2, -1, 0, 2, -1, 0, 1, -2, 1)^\T$ are two integral eigenvectors of $A(G)$ and $A(H)$ corresponding to the zero eigenvalue.
Clearly  $||\eta||/|| \xi||$ equals $\sqrt{3}$, which is irrational. By Lemma \ref{intirration},  there is no rational  orthogonal matrix $Q$ such that $A(G)=Q^\T A(H)Q$.

It is known that the GM-switching method corresponds to a special kind of rational orthogonal matrices. This method has been generalized recently by choosing different rational orthogonal matrices; see \cite{abiad2012ELJC,wang2019LAA,qiu2020LAA}. But since the orthogonal matrices in all these methods are rational,  the cospectral pair $G$ and $H$ shown in Figure \ref{ex16} can not be generated in this way. We remark that the same conclusion also holds for the pair in Figure 1 by considering the common simple eigenvalue $\lambda=0$.
\section*{Declaration of competing interest}
There is no conflict of interest.
\section*{Acknowledgments}
The authors are grateful to the referees for their careful reading and valuable comments. This work is supported by the	National Natural Science Foundation of China (Grant Nos. 12001006, 11971406) and the Scientific Research Foundation of Anhui Polytechnic University (Grant No.\,2019YQQ024).

\end{document}